%% file: mainrev.tex
\documentclass[12pt,a4paper]{amsart}
\usepackage[T1]{fontenc}
\usepackage[utf8]{inputenc}
\usepackage{amsmath,amsfonts,amssymb,amscd}
\usepackage{mathrsfs}
\usepackage{float,enumitem}
\usepackage{hyperref}
\usepackage{ifthen}
\usepackage{tikz}
\usepackage{tikz-cd}
\usepackage{graphicx}

\input{macros}

\newtheorem{theorem}{Theorem}[section]
\newtheorem{lemma}[theorem]{Lemma}
\newtheorem{corollary}[theorem]{Corollary}
\newtheorem{proposition}[theorem]{Proposition}

\newtheorem{remark}[theorem]{Remark}

\newtheorem{theostar}{Theorem}

\numberwithin{equation}{section}

\title[Homotopy type]{Topological complexity of arithmetic locally symmetric spaces}

\author{Miko\l{}aj Fr\k{a}czyk}
\address{Jagiellonian University, Faculty of Mathematics and Computer Science, ul. {\L}ojasiewicza 6, 30-348 Krak{\'o}w, Poland }
\email{mikolaj.fraczyk@uj.edu.pl}
\author{Sebastian Hurtado}
\address{Yale University, Department of Mathematics, 10 Hillhouse Ave, New Haven, CT 06511}
\email{sebastian.hurtado-salazar@yale.edu}
\author{Jean Raimbault}
\address{Institut de Mathématiques de Marseille, UMR 7373, CNRS, Aix-Marseille Université}
\email{jean.RAIMBAULT@univ-amu.fr}


\begin{document}
\maketitle
\begin{abstract}
 We prove that any arithmetic locally symmetric space is  homotopy equivalent to a simplicial complex where the number of simplices is bounded linearly in the volume of the space. This settles a well-known conjecture of Gelander. The main technical ingredient, which is of independent interest, is a strengthened version of the Margulis' collar lemma for arithmetic locally symmetric spaces based on the height gap theorem of Breuillard, in which the Margulis constant is made linear in the degree of the trace field of the lattice. 
\end{abstract}

\setcounter{tocdepth}{1}
\tableofcontents
\setcounter{tocdepth}{2}
\pdfbookmark[0]{Table of contents}{tableofcontents}
\section{Introduction}

\subsection{Topological complexity} Let $G$ be a semi-simple real Lie group with finite center, and let $X$ be the associated symmetric space where $G$ acts by isometries. We examine the relation between the topological complexity of finite volume manifold quotients $\Gamma \bs X,$ where $\Gamma$ is a torsion-free lattice in $G$, and their Riemannian volume $\vol(\Gamma\bs X).$ The idea that the volume should control the topology of locally symmetric spaces can be traced back to Gauss-Bonnet's formula which, in the case of compact hyperbolic surfaces, implies that the volume completely determines the topology. 

We cannot hope for such a strong relation in general. Still, one can ask how does the volume control the topological complexity of the space, which for example can be measured by estimating the minimal number of simplices in a triangulation of the space. The Gelander conjecture \cite[Conjecture 1.3]{Gelander1} states that any locally symmetric manifold of finite volume should be homotopy equivalent to a bounded degree simplicial complex where the number of simplices is bounded by the volume, thereby predicting a linear relation between the volume and the topological complexity. 

For non-compact arithmetic manifolds Gelander proved his conjecture in the very paper where he stated it \cite[Theorem 1.5(1)]{Gelander1}.  For compact arithmetic manifolds, it is well-known that a positive answer to Lehmer's problem on Mahler measures would immediately imply the conjecture. However, Lehmer's question remains wide open to this day, so the success of this approach is conditional on a major breakthrough in number theory (assuming the answer is indeed positive). As a consolation, Dobrowolski's bounds on Mahler measure imply a weaker version of the conjecture, as proven in \cite{gelander2021bounds}. For compact locally symmetric manifolds not covered by $\mathbb H^3, \mathbb H^2\times \mathbb H^2$ or $\SL_3(\mathbb R)/\SO(3)$(not necessarily arithmetic) a weak version of the Gelander conjecture is proven in \cite[
Theorem 11.2]{Gelander2}: the fundamental group is isomorphic to the fundamental group of a simplicial complex of complexity bounded as in the conjecture.

For not necessarily arithmetic locally symmetric spaces (and more generally manifolds of non-positive curvature) there are results towards Gelander's conjecture, which establish bounds for the homotopy type of the thick part. For manifolds of negative curvature a very strong result in this direction is given in \cite[Theorem 1.4]{Bader_Gelander_Sauer}. While it cannot settle the conjecture in full, this is strong enough to yield a lot of information on the topology of the manifold itself. However, the extension of the methods of \cite{Bader_Gelander_Sauer} to higher-rank symmetric spaces seems to present outstanding difficulties, as discussed in \cite[Section 1.5]{Bader_Gelander_Sauer}. 

The first author proved Gelander's conjecture for compact arithmetic hyperbolic 2-- and 3--manifolds in \cite{Fraczyk}, as a by-product of effective bounds of the volume of the thin part, which in turn followed from careful estimates on the geometric side of Selberg's trace formula. The trace formula approach presents substantial technical difficulties in the general case as the isometry group of the symmetric space grows more complicated, which was partly the reason why \cite{Fraczyk} dealt only with hyperbolic $2$-- and $3$--manifolds. In this work, we introduce a new method to control the thin part of completely general compact arithmetic orbifolds, allowing us to bound the volume of the thin part essentially bypassing the use of trace formula. Since the non-compact case was already handled in \cite{Gelander1}, these bounds enable us to provide a full solution to Gelander's conjecture for all arithmetic locally symmetric spaces.

\begin{theostar} \label{Gelander_conj}
  There are constants $A,B$ dependent only on the symmetric space $X$, such that every closed arithmetic manifold $M=\Gamma\bs X$ is homotopy equivalent to a simplicial complex $\mathcal N$ with at most $A \vol(M)$ simplices, where every vertex is incident to at most $B$ simplices. 
\end{theostar}

An important consequence of this result is the following estimate on torsion homology.

\begin{theostar} \label{torshom}
  There is a constant $C$ depending only on $X$ such that for any arithmetic locally-$X$ manifold $M=\Gamma\bs X$ of finite volume and for $0 \le i \le \dim(X)$ we have 
  \[
  \log |H_i(M, \ZZ)_{\mathrm{tors}}| \le C \vol(M)
  \]
  where for an abelian group $A$ we denote by $A_{\mathrm{tors}}$ its torsion subgroup. 
\end{theostar}

Similar estimates for Betti numbers were well-known in the setting of locally symmetric spaces by the work of Ballmann--Gromov--Schroeder \cite{BallGro} but these analogous bounds for torsion are new, especially in the higher rank cases. The implication from Theorem \ref{Gelander_conj} to Theorem \ref{torshom} is standard: see for instance \cite[Section 5.2]{Bader_Gelander_Sauer} which outlines a general argument deducing the bounds on homology from the existence of a triangulation with bounded degrees. 

We note that in the case where $X \neq \mathbb H^3$ is negatively curved, the results of \cite[Theorem 1.2]{Bader_Gelander_Sauer} are sufficient to imply an estimate on torsion homology similar to that of Theorem \ref{torshom}.

\subsection{Arithmetic Margulis lemma.} The key new ingredient in our arguments is the following extension of the classical Margulis lemma. In what follows $d\colon X\times X\to\mathbb R_{\geq 0}$ is the Riemannian distance on the symmetric space $X$. 
\begin{theostar}\label{MargLemmaIntro}
There exists a constant $\varepsilon=\varepsilon_X$ with the following property. Let $\Gamma\subset G$ be an arithmetic lattice and let $k=\mathbb Q(\{\tr {\rm Ad}(\gamma)\,|\, \gamma\in \Gamma\})$ be the adjoint trace field of $\Gamma$ (see \ref{arithlattices}). For any $x\in X$ the group generated by $\{\gamma\in \Gamma\,|\, d(x,\gamma x)\le [k:\mathbb Q]\varepsilon\}$ is virtually nilpotent.
\end{theostar}
Note that this lemma effectively allows us to control the $[k:\mathbb Q]\varepsilon-$thin part of the quotient $\Gamma\bs X,$ while the traditional version of Margulis lemma handles only the $\varepsilon$-thin part. Theorem \ref{MargLemmaIntro} can be proved fairly quickly using Breuillard's height gap theorem \cite{Breuillard} (See \cite{chen2021height} for a recent short proof). Despite the shortness of the argument, the present work (as far as authors know) is the first time Breuillard's height gap was used in the study of locally symmetric spaces. Since the first version of this work appeared, the arithmetic Margulis lemma was used in  \cite{fisher2022new}, \cite{raimbault2022coxeter} to give simple proofs of the finiteness of arithmetic reflection reflection groups and we expect new applications to follow.   

To construct the simplicial complex in Theorem \ref{Gelander_conj} we use the same technique as in \cite{Gelander1}, namely taking the nerve of a good cover by balls; the proof consists in estimating the number of centers we need for such a cover in terms of the volume of $M = \Gamma \bs  X$. This set of centers is going to be a sufficiently dense and sufficiently separated subset of points in $M$, the difficulty is that the density and separation must depend on the local injectivity radius. However a straightforward argument shows that for a compact manifold we can always use a set with $O_\eps\left(\vol(M) + \int_{M_{\le \eps}} \InjRad_M(x))^{-\dim(M)} dx \right)$ points (where $\eps$ is any positive number). To guarantee that this is $O(\vol (M))$ we need estimates for the volume of the thin part\footnote{if $R > 0$, we denote the $R$-thin part by $M_{\le R}$, which is the set of points where the injectivity radius is smaller than $R$.} $M_{\le \eps}$ of an arithmetic locally symmetric space in terms of the degree of its field of definition. This is given by the following result. This is given by the following result. We note that while the construction of the good cover is only done here for manifolds and not orbifolds, the arithmetic Margulis lemma and the results below apply to all arithmetic lattices, torsion-free or not. 

\begin{theostar} \label{Main_thinpart}
  Let $G$ be a semisimple Lie group with associated symmetric space $X$. There are positive constants $c=c_X,\eta=\eta_X$ such that any irreducible arithmetic lattice $\Gamma\subset G$ with trace field $k$ satisfies 
  \[
  \vol((\Gamma\bs X)_{\le \eta[k:\QQ]}) \le \vol(\Gamma\bs X)e^{-c[k:\QQ]}.
  \]
\end{theostar}


Theorem \ref{Main_thinpart} is proved using the Arithmetic Margulis Lemma and new estimates on orbital integrals (which might be of independent interest) that allow us to estimate the growth of the volume of neighborhoods of thin collars. 

\begin{theostar}\label{thm-OEintro}
  Let $G$ be a semisimple Lie group with finite center and let $\Gamma$ be a uniform lattice in $G$. There exists  $C, \delta > 0$ (depending only on $G$) such that for every $R_1, R_2 \geq 0$ and every non-central $\gamma \in \Gamma$: 
  \[
  \int_{G_\gamma\backslash G}1_{B(R_1)}(g^{_1}\gamma g)dg < C e^{-\delta R_2} \int_{G_\gamma\backslash G}1_{B(R_1 + R_2)}(g^{-1}\gamma g)dg, 
  \] where $G_\gamma$ is the centralizer of $\gamma$ in $G$. 
\end{theostar}

The proof of Theorem \ref{thm-OEintro} is based on the existence of a uniform spectral gap for the family of regular representation representations $L^2(H\bs G)$ of $G$, where $H$ is a centralizer of a non-central element of $G$. 

\subsection*{Remarks}
We also have applied the Arithmetic Margulis Lemma to give estimates on homology growth and Benjamini-Schramm convergence of sequences or arithmetic locally symmetric spaces in \cite{frachurtraim}, results which were originally part of the first version of this work.


\subsection*{Acknowldedgments} We are grateful to Misha Belolipetsky, Tsachik Gelander and David Vogan for their comments on a previous version of this work, and to Nir Lazarovich for his help with complexity. M.~F. and S.~H. thank the members of UIC Groups and Dynamics seminar. S.~H. was supported by the Sloan Fellowship foundation. J.~R. was supported by grants AGIRA - ANR-16-CE40-0022 and ANR-20-CE40-0010 from the Agence Nationale de la Recherche. M. Fr\k{a}czyk was supported by the Dioscuri programme initiated by the Max Planck Society, jointly managed with the National Science Centre in Poland, and mutually funded by Polish the Ministry of Education and Science and the German Federal Ministry of Education and Research.


\section{Preliminaries on arithmetic lattices}

\subsection{Arithmetic lattices}\label{arithlattices}

Let $G$ be an adjoint Lie group without compact factors. In this paragraph we review the standard construction of arithmetic lattices in $G$. Recall first that if $k$ is a number field and $\G$ is a semisimple $k$-group then an arithmetic subgroup of $\G(k)$ is any subgroup which has finite index inside the stabiliser of a lattice in the $k$-Lie algebra of $\G$. Now let $k$ be a number field and let $\G$ be a $k$-group such that $\G(k\otimes_\QQ \RR)\simeq G\times G'$, where $G'$ is compact. 
Then, by the Borel--Harish-Chandra theorem \cite{BHC62}, the image in $G$ of an arithmetic subgroup of $\G(k)$ is a lattice in $G$, and it is called an arithmetic lattice\footnote{This definition of arithmetic lattices uses adjointness of $G$ through \cite[Proposition 1.4]{Borel_Prasad}; without this hypothesis one has to allow passing to slightly larger subgroups. }. Recall that $\Gamma$ is irreducible when the Lie algebra of $\G$ is simple over $k$. 

If $\Gamma$ is an arithmetic lattice in $G$, then the work of Vinberg \cite{Vinberg_trace} allows to recover the field $k$ and the group $\G$ 
from $\Gamma$ itself. Namely, $k$ is the field generated by the adjoint traces of elements of $\Gamma$, and then $\ad(\Gamma)$ is contained in a minimal $k$-group which is $k$-isogenous to $\G$. We will call $k$ the {\em adjoint trace field} of $\Gamma$ and $\G$ its {\em enveloping group}. Slightly more is true: it follows from the definition that $\tr(\ad(\gamma))$ is an algebraic integer in $k$ for any $\gamma \in \Gamma$, and the following fact is a quick consequence.

\begin{lemma}
  For any arithmetic lattice $\Gamma$ in $G$, any semisimple element $\gamma \in \Gamma$ and any maximal torus $T$ of $G$ containing $\gamma$ the following is true. If $\lambda$ is a root of (the complexification of) $(G, T)$ then $\lambda(\gamma)$ is an algebraic integer. 
\end{lemma}
\begin{proof}
Choose a semisimple $k$-group $\G$ such that $\G(k\otimes_{\QQ}\RR)=G\times G'$ with $G'$ compact and $\Gamma$ is an arithmetic subgroup of $\G(k)$. There is a rational representation $\rho\colon \G\to \GL_n(k)$, with finite kernel, such that $\rho(\Gamma)\subset \GL_n(\mathfrak o_k)$. Let $\Sigma$ be the set of $\overline k$-weights of $T$ in the representation $\rho.$ For any $\xi\in \Sigma$ and $\gamma\in T(k)\cap \Gamma$ the value $\xi(\gamma)$ is an eigenvalue of the matrix $\rho(\gamma)\in \GL_n(\mathfrak o_k)$, hence an algebraic integer. On the other hand, by \cite[\S 31.1]{BorelBook}, we know that the set of weights of $T$ in $\rho$ generates a finite index subgroup of all characters of $T$. It follows that for any character $\eta$ of $T$, the value $\eta(\gamma)$ is an algebraic integer, in particular it will be true for the roots. 
\end{proof}


\subsection{Weil height and minimal displacement}

Recall that if $G$ is a semisimple Lie group, then there is an associated symmetric space $X = G/K$, where $K$ is a maximal compact subgroup. The space $X$ carries a canonical non-positively curved $G$-invariant Riemannian metric induced by the Killing form; we denote by $d_X$ its Riemannian distance.

If $g \in G$ is semisimple, there is a flat Euclidean subspace $F \subset X$ which is $g$-invariant and such that $g$ acts on $F$ as a translation. The value $d_X(x, gx)$ for any $x \in X$ is called the translation distance of $g$ and we will denote it by $\ell(g)$; by the CAT(0)-property of $d_X$ it equals the minimal displacement of $g$ on $X$. 

\medskip

Our goal here is to estimate the translation distance of elements of arithmetic lattices in terms of the Weil height, which we will now define. Let $a$ be an algebraic integer and let $K$ be a number field containing $a$. Let $\iota_1, \ldots, \iota_{[K:\QQ]}$ be all the embeddings of $K$ into $\CC$ (including complex conjugates). The Weil height of $a$ of an algebraic integer is given by
\begin{equation} \label{defn_Weil_ht}
  h(a) = \frac 1{[K:\QQ]} \sum_{i=1}^{[K:\QQ]} \log^+|\iota_i(a)|,
\end{equation}
which does not depend on the choice of the field $K$. Here $\log^+(t) = \max(0, \log(t))$. 

Let $\Gamma$ be an arithmetic lattice in an adjoint group $G$, with trace field $k$ and enveloping group $\G$. Let $\gamma \in \G(k)$ be a semisimple element. Choose a maximal $k$-torus $\T\subset \G$ containing $\gamma$ and let $\Phi=\Phi(\G,\T)$ be the associated root system. All the values $\lambda(\gamma)$ are algebraic integers. We define the Weil height $h(\gamma)$ by :
\begin{equation} \label{defn_Weil_ht_ss}
  h(\gamma) = \sum_{\lambda \in \Phi} h(\lambda(\gamma)).
\end{equation}
Now we state the main result about arithmetic lattices that we will be using in the sequel. The following statement seems to be well-known to experts (see e.g. \cite{Lapan_Linowitz_Meyer,GelaVol}) but it is hard to locate in the literature in the form we need, so we give a short argument using a computation in the former reference. We phrase it in terms of Weil height rather than Mahler measure since the former occurs more naturally in the proof and will be more convenient to use in the sequel. 

\begin{proposition} \label{p-MahlerDisp}
  Let $G$ be an adjoint Lie group without compact factors and $X$ its symmetric space. There exist constants $a_X, A_X$ such that for any irreducible arithmetic lattice $\Gamma$ in $G$ with trace field $k$ and any semisimple element $\gamma \in \Gamma$ we have
  \begin{equation} \label{eq-MahlerDisp}
    a_X \le \frac{\ell(\gamma)}{[k:\QQ] h(\gamma)} \le A_X.
  \end{equation}
\end{proposition}

\begin{proof}
  Let $\G$ be the enveloping $k$-group of $\Gamma$. For any embedding $\sigma$ of $k$ into $\CC$ let $k_\sigma$ be the completion of $k$ with respect to the induced valuation. Let $\sigma_1,\ldots, \sigma_d$ be the list of embeddings $k\hookrightarrow\mathbb C,$ with one representative for each pair of complex conjugates. We note that $\G(k\otimes_\QQ \RR)\simeq \prod_{i=1}^d \G(k_{\sigma_i}).$ Let $r$ be the number of simple real or complex factors of $G$. By the choice of $\G$, we can order the embeddings in such a way that $G\simeq \prod_{i=1}^r \G(k_{\sigma_i})$ and $\G(k_{\sigma_j})$ are compact for all $d\geq j>r$. Since there is at least one embedding $\sigma$ of $k$ into $\CC$ such that $\G(k_\sigma)$ is a simple factor of $G$, we have that $\dim\G$ is bounded by the dimension of $G$. Let $n = \dim \G$ and let  $
  \rho \colon  \G\to \SL_n $
  be the adjoint representation of $\G$. 

   For each $i=1,\ldots, d$ let $\rho_i\colon \G(k_{\sigma_i})\to \SL_n(k_{\sigma_i})$ be the map induced by $\rho$. 
 The map 
 \[\rho\colon \G(k\otimes_\QQ \RR)=\prod_{i=1}^d \G(k_{\sigma_i})\to \prod_{i=1}^d \SL_n(k_{\sigma_i})=\SL_n(k\otimes_\QQ \RR)\] induces an isometric embedding of $X$ into the symmetric space $Y=\prod_{i=1}^d Y_i$, where $Y_i$ is the symmetric space of $\SL_n(k_{\sigma_i})$. 

  We now fix a semisimple element $\gamma \in \Gamma$, a maximal torus $\mathbf T$ containing $\gamma$ and let $\Phi=\Phi(\G,\mathbf T)$ be the associated root system. Denote by $\ell_i(\gamma)$ the translation length of $\rho_{\sigma_i}(\gamma)$ on $Y_i$. We note that for each $j>r$, the translation length $\ell_i(\gamma)$ is zero because $\G(k_{\sigma_i})$ is compact.  By Cauchy--Schwarz, the translation length $\ell(\rho(\gamma))=\sqrt{\sum_{i=1}^r \ell_i(\gamma)^2}$ on $Y$ is asymptotic to $\sum_{i=1}^r \ell_i(\gamma)$, with constants depending only on $r$, which itself is bounded by $\dim(X)$. Let us first estimate each $\ell_i(\gamma)$. We fix an extension of $\sigma_i$ to an embedding $\tau_i$ of the field $k(\lambda(\gamma)\mid \lambda\in \Phi)$ (for shortness we will denote this field by $k(\Phi(\gamma))$) into $\CC$. Note that the numbers $\tau_i(\lambda(\gamma))$ for $\lambda\in \Phi$ are exactly the roots of the characteristic polynomial of $\rho_{\sigma_i}(\gamma)$. In particular, the set $\{\tau_i(\lambda(\gamma)), \lambda\in \Phi\}$ does not depend on the choice of the extension $\tau_i$. By \cite[Theorem 2.3]{Lapan_Linowitz_Meyer} we get that\footnote{Be careful that the definition in \cite[Equation (2.1)]{Lapan_Linowitz_Meyer} for the Mahler measure is unconventional as they do not restrict to polynomials with $\ZZ$ coefficients. }
  \[
  \ell_i(\gamma) \asymp_n \sum_{\lambda\in\Phi} \log^+|\tau_i(\lambda(\gamma)|
  \]
  with implicit constants depending only on $n$. It follows that
  \begin{equation}\label{est_in_Y}
    \ell(\rho(\gamma)) \asymp_{X} \sum_{i=1}^r \ell_i(\gamma) \asymp_X \sum_{i=1}^r \sum_{\lambda \in \Phi} \log^+ |\tau_i(\lambda(\gamma))| 
  \end{equation}
  For $j>r$ and any embedding $\tau\colon k(\Phi(\gamma))\to\CC$ extending $\sigma_j$, we have $|\tau(\lambda(\gamma))| = 1$ for all $\lambda \in \Phi$, since $\G(k_{\sigma_j})$ is compact. Therefore, the last sum is equal to $$\sum_{i=1}^d \sum_{\lambda \in \Phi} \log^+ |\tau_i(\lambda(\gamma))|$$ 
  As $[k(\Phi(\gamma)):k] \le n$ we get, according to the definitions \eqref{defn_Weil_ht} and \eqref{defn_Weil_ht_ss}, that 
  \[
  [k:\QQ] h(\gamma) \le \sum_{i=1}^d \sum_{\lambda \in \Phi} \log^+ |\tau_i(\lambda(\gamma))| \le n \cdot [k:\QQ] h(\gamma)
  \]
  Since the image of $X$ is totally geodesic in $Y$, $\ell(\rho(\gamma))$ equals $\ell(\gamma)$ up to a multiplicative constant depending only on $X$. This, together with \eqref{est_in_Y} finishes the proof. 
\end{proof}


\subsection{Dobrowolski's bounds and injectivity radius}\label{sec-injrad}

Let $d(g,h):=d_X(gK,hK)$. It is a bi--$K$--invariant, left $G$--invariant semi-metric on $G$. Let 
\[
\B(R) = \{ g\in G : d_X(K,gK) \leq R \}=\{g\in G : d(g,1)\leq R\}.
\]
One can think of $\B(R)$ as a ball in $G$, although it only corresponds to the semi-metric $d$ on $G$. Let $\Lambda\subset G$ be a discrete subgroup of $G$ and let $x = \Lambda g K\in \Lambda\bs X$ be a point in the locally symmetric space $M = \Lambda\bs X$. The {\em injectivity radius} of $M$ at $x$ is defined as 
\begin{align*}
  \InjRad_{\Lambda\bs X}(x) &= \sup\left\{ R\geq 0\,|\, d_X(gK, \gamma gK)\geq \frac R 2 \textrm{ for every } \gamma\in \Lambda \setminus Z(\Lambda) \right\} \\
  &=\sup\{R\geq 0\,|\, \Lambda^g\cap \B(2R)=Z(\Lambda)\}.
\end{align*}
If $M$ is a manifold (that is, if $\Lambda$ is torsion-free) then this is the same as the usual definition of injectivity radius. In general, this definition also takes into account the singular locus of the orbifold $M$. We also recall that the global injectivity radius is defined by 
\[
\InjRad_M := \inf_{x \in M} \InjRad_M(x). 
\]
As an application of Proposition \ref{p-MahlerDisp} together with bounds on Mahler measure due to Dobrowolski one obtains the following estimate on the injectivity radius of arithmetic manifolds. The estimate below follows immediately from \cite[Prop. 1.4]{GelaVol}, we provide a short proof for reader's convenience.

\begin{proposition} \label{dobrowolski_bound_injrad}
  Let $G$ be an adjoint Lie group without compact factors and $X$ its symmetric space. There exists a constant $c$ depending only on $G$ such that if $\Gamma$ is a cocompact irreducible arithmetic lattice in $G$ with trace field $k$ and $[k:\QQ] \geq 2$, then for any semisimple non-compact $\gamma \in \Gamma$ we have
  \[
  \ell(\gamma) \ge \frac{c}{(\log [k:\QQ])^3}.
  \]
  In particular, if $M = \Gamma \bs X$ is compact and $\Gamma$ is torsion-free, then 
  \[
  \InjRad_M \ge \frac{c}{(\log [k:\QQ])^3}.
  \]
\end{proposition}

\begin{proof}
  The algebraic integers $\lambda(\gamma)$ have degree at most $[k:\QQ] \cdot \dim(G)$ over $\QQ$. By \cite[Theorem 1]{Dobro79} it follows that $m(\lambda(\gamma)) \ge \frac 1{2(\log[k:\QQ]+\log\dim(G))^3}$ for all $\lambda$ such that $\lambda(\gamma)$ is not a root of unity. Since $\gamma$ is non-compact there exists at least one such $\lambda$ and it follows that
  \[
m(\gamma) \ge \frac 1{2(\log[k:\QQ]+\log\dim(G))^3} \gg_G \frac 1{\log[k:\QQ]^3}, 
  \]
  and the result follows by \eqref{eq-MahlerDisp}. 
\end{proof}


\section{Arithmetic collar lemma}

\begin{theorem}\label{AML}
Let $G$ be a real semi-simple Lie group. There exists a constant $\varepsilon_G>0$ with the following property. Let $\Gamma\subset G$ be an irreducible arithmetic lattice with trace field $k$. Let $x\in X$. Then, the subgroup generated by the set 
\[\{\gamma\in\Gamma\, |\, d(x,\gamma x)\leq \varepsilon_G [k:\QQ]\}\] is virtually nilpotent. Moreover, if $\Gamma$ is uniform then this subgroup is virtually abelian.  
\end{theorem}

We will deduce \ref{AML} from the height gap theorem of \cite{Breuillard} in the following form.

\begin{theorem}[{\cite[Cor 1.7]{Breuillard}}]\label{thm-Breuillard} There are constants $N_1 = N_1(d) \in \mathbb N, 
\varepsilon_1 = \varepsilon_1(d) > 0$
such that if $F$ is any finite subset of $\GL_d(\overline\QQ)$ containing $1$ and generating a
non-virtually solvable subgroup, then we may find $a\in F^{N_1}$ and an eigenvalue $\lambda$
of $a$ such that $h(\lambda) > \varepsilon_1$.
\end{theorem}

We will also need the following well-known lemma. 

\begin{lemma}\label{lem-anisolvable}
Let $\bH$ be a connected solvable subgroup of a linear semi-simple group $\bG$. Suppose that both $\bH,\bG$ are defined over $k$ and that $\bG$ is anisotropic over $k$. Then $\bH$ is abelian.
\end{lemma}

\begin{proof}
Recall that a semi-simple algebraic group $\G$ over $k$ is anisotropic if one of the following equivalent conditions is satisfied:  
\begin{enumerate}
\item  There are no non-trivial $k$-rational maps $\mathbb G_m\to \G$ 
\item There are no rational unipotent subgroups of $\mathbf G$.
\end{enumerate} 
Let $\H$ be as in the statement and let $\mathbf N$ be the unipotent radical of $\bH$. Then, by \cite[10.6]{BorelBook}, the derived subgroup $[\bH,\bH]$ is contained in $\mathbf N$. However, $\bN$ is a $k$-rational unipotent subgroup of $\bG$ so it must be trivial. We conclude that $\bH$ is abelian. 
\end{proof}

\begin{proof}[Proof of Theorem \ref{AML}]
Let $\eps > 0$, $x\in X$ and put
\[
F=\{\gamma\in\Gamma\, |\, d(x,\gamma x)\leq \eps [k:\QQ]\}. 
\]
We first want to determine an $\eps > 0$ depending only on $G$ such that (for any $x$) the subgroup generated by $F$ is virtually solvable. 

Thus, we fix $\eps>0$, to be determined in the course of the argument. We assume that $\langle F\rangle$ is not virtually solvable, and we want to reach a contradiction through Breuillard's gap theorem. From Theorem \ref{thm-Breuillard} we get that for $\eps$ small enough (depending only on $d = \dim(\G)$) there exists a $\gamma \in F^{N_1}$ and an eigenvalue $\lambda$ of $\ad(\gamma)$ such that $h(\lambda) > \eps_1$. Now the eigenvalue $\lambda$ is equal to a root evaluated at $\gamma$, and the Weil height of algebraic numbers is positive, so we have that 
\[
h(\gamma) \ge \eps_1
\]
by the definition \eqref{defn_Weil_ht_ss}. 
It then follows from Proposition \ref{p-MahlerDisp} that  
\[
\ell(\gamma) > a_X \eps_1[k:\QQ],
\]
where $a_X$ depends only on $G$. On the other hand, 
\[
\ell(\gamma) \le d(x, \gamma x) < N_1\eps[k:\QQ]
\]
as $\gamma \in F^{N_1}$. So if $\eps \le \tfrac{a_X\eps_1}{N_1}$ we get our contradiction, and we observe that the constant $\tfrac{a_X\eps_1}{N_1}$ depends only on $G$. 

\medskip 

It remains to prove that $\langle F\rangle$ is virtually abelian if $\Gamma$ is cocompact and virtually nilpotent in general. In the case where $\Gamma$ is cocompact this is an immediate consequence of Lemma \ref{lem-anisolvable}, as $\Gamma$ is cocompact if and only if $\G$ is anisotropic over $k$ (this would also follow from the Solvable Subgroup Theorem of \cite{BriHa}). 

If $\Gamma$ is not cocompact, then we may assume that $[k:\QQ]=1$ (see \cite[Lemma 5.2]{Gelander1}). In this case the theorem reduces to the classical Margulis collar lemma for symmetric spaces. 
\end{proof}


\section{Uniform growth of orbital integrals}

\subsection{Orbital integrals} \label{sec-orbint} 
Let $f \in C_c(G)$ be a compactly supported function. Recall that for a semisimple element $\gamma \in G$ the orbital integral associated with $(f, \gamma)$ is defined as follows: 
\[
O(\gamma, f) = \int_{G_{\gamma} \bs G} f(g^{-1}\gamma g) dg.
\]
While is not obvious that this integral converges whenever $\gamma$ is semi-simple, it follows for instance from the fact that the map $G_\gamma\bs G$, $g\mapsto g^{-1}\gamma g$ is proper, which is proven in \cite[Theorem 27 on p. 215, Theorem 17 on p. 211]{Varadarajan}. 

\subsection{Uniform growth}
The main result in this section is the following. 

\begin{theorem}\label{thm-UniformGrowth}
  Let $G$ be a semisimple Lie group with finite center and let $\Gamma$ be a uniform lattice in $G$. There exists  $C, \delta > 0$ (depending only on $G$) such that for every $R_1, R_2 \geq 0$ and every non-central $\gamma \in \Gamma$: 
  \[
  \mathcal O(\gamma, 1_{B(R_1)}) < C e^{-\delta R_2} \mathcal O(\gamma, 1_{B(R_1 + R_2)}).
  \]
\end{theorem}

The proof is based on the existence of uniform spectral gap for the action of $G$ by conjugation on each non-central semi-simple conjugacy class. 

\begin{lemma}\label{lem:UnifSG} Let $G$ be a semi-simple group without compact factors. There is a positive constant $\delta>0$ and a non-negative smooth bi--$K$--invariant symmetric function $\varphi\colon G\to \mathbb R$ supported on $\B(1)$ such that  $\int_G \varphi(g)dg=1$ and the following property holds. For every proper closed reductive subgroup $H\subset G$ which is a centralizer of a semisimple element and every $f\in L^2(H\bs G)$ we have 
\[ \|f\ast \varphi\|_2\leq e^{-\delta} \|f\|_2. \]
\end{lemma}
\begin{proof}
$G$ decomposes as a product of simple groups $G=G_1\times\ldots\times G_n$. Since $H$ is a centralizer, we will have a similar decomposition $H=H_1\times\ldots\times H_n$, where $H_n$ is a centralizer of a semi-simple element in $G_i$ and at least one of $H_i$'s is a proper subgroup of $G_i$. It is clear that the lemma will follow once we prove it for each factor separately. We can therefore assume that $G$ is a simple non-compact Lie group. Choose any function $\varphi$ satisfying the assumption of the lemma. Let $M$ be the operator of the right convolution by $\varphi$. We need to show that $\|M\|_{L^2(H\bs G)}\leq e^{-\delta}$, i.e. that $\|M\|_{L^2(H\bs G)}$ is uniformly bounded away from $1$. Suppose $G$ has Kazhdan's property (T) and consider the unitary representation $\pi=\bigoplus_{H}L^2(H\bs G)$, where $H$ runs over the conjugacy classes of proper reductive subgroups of $G$. This representation has no fixed vectors, so by property (T) it has no almost invariant vectors. Since the support of $\varphi$ generates $G$, we deduce that $\|M\|_\pi=\sup_{\|v\|=1}\langle v,\pi(\varphi)v\rangle <1.$ 

It remains to treat the cases of non-compact simple Lie groups without property (T). These are, up to isogeny, $G = \SO(n,1),\SU(n,1),n\geq 2$ (see \cite[3.5.4]{BekkaValette}). The compact factors of $H$ do not affect the norm of the convolution by $\varphi$, so we can assume that the pair $(G,H)$ is (i) $(\SO(n,1),\SO(m,1)), 0<m<n$, (ii) $(\SU(n,1),\SO(m,1)), 0<m\leq n$ or (iii) $(\SU(n,1),\SU(m,1))$, with $0<m<n$. 
Let $\Lambda$ be a uniform lattice in $H$. Since $L^2(H\bs G)$ is a sub-representation of $L^2(\Lambda\bs G)$, it is enough to show that $M$ has a uniform spectral gap on $L^2(\Lambda\bs G)$. The function $\varphi$ is bi--$K$-invariant so this will follow from the uniform spectral gap for the Laplace operator acting on $L^2(\Lambda\bs X)$. Let $\lambda_0(\Lambda\bs X)$ be the lowest eigenvalue of the Laplace operator acting on $L^2(\Lambda\bs X)$. Let $\delta(\Lambda)$ be the critical exponent of $\Lambda$ in $G$, it is equal to $m,m,2m$ is cases (i),(ii),(iii) respectively. By \cite[Thm 4.2]{Cor1} (see also \cite{Els1, Els2, Els3, Pat1, Sul1}) we have 
\[
\lambda_0(\Lambda\bs X)=\begin{cases} \frac{1}{4}(\dim \partial X)^2 & \text{ if }0\leq \delta(\Lambda)\leq \frac{1}{2}\dim \partial X,\\
\delta(\Lambda)(\dim\partial X-\delta(\Lambda)) &\text{ if } \frac{1}{2}\dim\partial X< \delta(\Lambda)\leq \dim\partial X.
\end{cases}
\]
It follows that $\lambda_0(\Lambda\bs X)$ is bounded away from $0$ uniformly in $m$. The lemma is proved,
\end{proof}

We are ready to prove Theorem \ref{thm-UniformGrowth}. 
\begin{proof}[Proof Theorem \ref{thm-UniformGrowth}]
It is enough to prove that 
\begin{equation}\label{eqn:1step} \mathcal O(\gamma, 1_{B(R)}) \leq e^{-2\delta} \mathcal O(\gamma, 1_{B(R + 2)}),\end{equation}
for every $R\geq 0$. The theorem will follow by induction and the fact that  $\mathcal O(\gamma, 1_{B(R)})$ is non-decreasing in $R$. 

Consider the function $\psi_R\colon G_\gamma\bs G\to \RR_{\geq 0}$, given by $\psi_R(g)=1_{\B(R)}(g^{-1}\gamma g).$ Alternatively, we can describe $\psi_R$ as the characteristic function of the set 
\[\{G_\gamma g\in G_\gamma\bs G\,|\, d(g,\gamma g)\leq R\}, \] where $d\colon G\times G\to\mathbb R_{\geq 0}$ is the left $G$-invariant, right $K$-invariant semi-metric defined in Section \ref{sec-injrad}. The map $\Gamma_\gamma g\mapsto g^{-1}\gamma g$ is proper, so $\psi_R$ is compactly supported and we have \[\mathcal O(\gamma, 1_{\B(R)})=\|\psi_R\|_2^2.\] 
Observe that for any $g\in G_\gamma\bs G$ such that $d(g,\gamma g)\leq R$ and any $h\in B(1)$ we have 
\[d(gh,\gamma gh)\leq d(g,gh)+d(\gamma g,g)+d(\gamma gh,\gamma g)\leq R+2.\]
Let $\varphi$ be the function from the Lemma \ref{lem:UnifSG}. Since ${\rm supp\,}\varphi\subset \B(1)$, the above inequality implies that ${\rm supp\,}(\psi_R\ast \varphi)\subset {\rm supp\,}(\psi_{R+2}).$ In particular, $\psi_R\ast \varphi(g)=\psi_R\ast \varphi(g)\psi_{R+2}(g)$ for every $g\in G_\gamma \bs G$. By the Cauchy-Schwartz inequality we get that
\[
\|\psi_R\ast \varphi\|_2^2 \, \| \psi_{R+2}\|_2^2\geq \left(\int_{G_\gamma\bs G} \psi_R\ast \varphi(g)dg\right)^2
\]
Since $G$ preserves the measure on $G/G_\gamma$ and $\int_G f(g)dg=1$ we have that $\int_{G_\gamma\bs G} \psi_R\ast \varphi(g)dg = \int_{G_\gamma \bs G} \psi_R(g)dg$. As in addition $\int_{G_\gamma\bs G} \psi_R(g) dg = \mathcal O(\gamma, 1_{B(R)})$,  we obtain that 
\[
\|\psi_R\ast \varphi\|_2^2 \, \mathcal O(\gamma, 1_{\B(R+2)}) \geq \mathcal O(\gamma, 1_{B(R)})^2. 
\]
By Lemma \ref{lem:UnifSG} we have $\|\psi_R\ast \varphi\|_2^2\leq e^{-2\delta}\|\psi_R\|_2^2=e^{-2\delta}\mathcal O(\gamma, 1_{\B(R)})$. From this and the inequality above we infer that \[e^{-2\delta}\mathcal O(\gamma, 1_{\B(R+2)})\geq \mathcal O(\gamma, 1_{\B(R)}).\] This establishes (\ref{eqn:1step}).
\end{proof}


\section{Volume of the thin part}

Let $\Gamma$ be a lattice in a semisimple Lie group $G$ and let $R > 0$. For any $\gamma \in \Gamma, \gamma \not= \mathrm{Id}$, we define the following subset of $G$:
\[
S^{R}_{\gamma} := \{ g \in G  \text{ such that } d(g^{-1}\gamma g, 1) < R \},
\]
We note that if $\Gamma$ is uniform the volume of $\Gamma_\gamma \bs S_\gamma^R$ is always finite. We use the arithmetic Margulis lemma (Theorem \ref{AML}) to prove the following fact. 

\begin{proposition}\label{small intersection}
  There exists $\eps, m > 0$ only depending on $G$ such that if $\Gamma$ is a uniform irreducible arithmetic lattice in $G$ with trace field $k$, then every $g$ in $G$ lies in at most $[k:\QQ]^m$ sets of the form $S^{\eps [k:\QQ]}_\gamma$ where $\gamma \in \Gamma$. 
\end{proposition}

For the proof of this proposition we will also use the following consequence of Jordan's theorem on linear groups. 

\begin{lemma} \label{bdd_index_abelian}
  For every $n>0$ there exists an integer $A>0$ with the following property. Let $\Delta$ be a finitely generated subgroup of $\mathrm{GL}_n(\CC)$ which contains an abelian subgroup of finite index consisting of semisimple elements. Then, $\Delta$ has an abelian subgroup of index at most $A$. 
\end{lemma}

\begin{remark}
One can recover this result by combining \cite[10.11]{Wehrfritz} (or \cite{platonov1968some}) and \cite[3.6]{Wehrfritz} under the assumption that $\Delta$ consist only of semisimple elements. The former implies that $\Delta$ has a solvable finite index subgroup $\Delta'$ of index at most $O(n)$. By \cite[3.6]{Wehrfritz} we can find a $O(n)$-index subgroup $\Delta''\subset \Delta'$ which can be conjugated inside the subgroup upper triangular matrices. We can then conclude that $\Delta''$ is abelian , since $[\Delta'',\Delta'']$ consists only of unipotent elements and we assumed that all element of $\Delta$ are semisimple.  
\end{remark}

\begin{proof}
  Let $H_1$ be a semi-simple normal abelian torsion-free finite-index subgroup of $\Delta$. The group  $H_1$ can be conjugated into the group $D$ of diagonal matrices, so we may assume that $H_1 = \Delta \cap D$. In particular, $\Delta$ normalises the connected component of the Zariski closure $\bf T$ of $H_1$, which is a torus. It follows that $\Delta/H_1$ embeds into a product of groups of the form $S_{j_l} \wr \mathrm{PGL}_{i_l}(\CC)$ with $\sum_l i_l j_l = n$. Indeed, the quotient of the normaliser of $\bf T$  by $\bf T$ is a direct product of groups of this form. In particular, $\Delta/H_1$ is linear group of degree at most $n^2$ (since $\mathrm{PGL}_i$ has a linear representation of degree at most $i^2$). By Jordan's theorem , it follows that there exists a subgroup $H_1 \subset H_2 \subset \Delta$ such that $H_2/H_1$ is abelian and $|\Delta/H_2|$ is bounded by a constant $A_1$ depending only on $n$. 
  
  Let $H_3$ be the subgroup of $\Delta$ acting trivially by conjugation on $H_1$. The quotient $\Delta/H_3$ acts freely on $H_1$, so it embeds into the automorphism group of the character module $X^*({\bf T})$. The rank of $\bf T$ is bounded by $n$, so  $\Delta/H_3$ is isomorphic to a finite subgroup of $\GL_n(\ZZ)$. Therefore, $|\Delta/H_3|$ is bounded by a constant $A_2$ depending only on $n$. 
  
  Now the group $H_4:=H_2 \cap H_3$ contains $H_1$ and has index at most $A_1 A_2$ in $\Delta$. By construction $H_1\to H_4 \to H_4/H_1$ is a central extension of an abelian torsion free group by a torsion abelian group. We will prove that it is abelian, which finishes the proof of our lemma. Our argument relies on the following (probably well-known) fact: 
  
  \begin{lemma} A central extension of a torsion-free abelian group by a torsion abelian group is abelian.
  \end{lemma}

\begin{proof}

let $A \to B \to C$ be such an extension and choose any $c_1, c_2 \in C$ and lifts $b_1, b_2$ to $B$. We need only prove that $b_1$ and $b_2$ commute with each other. To do so, we note that the commutator $[b_1, b_2] = b_1b_2b_1^{-1}b_2^{-1}$ belongs to $A$ and depends only on $c_1, c_2$ and not on the choice of lifts. Moreover if $c_3 \in C$ and $b_3$ is a lift to $B$ we have:
  \begin{align*}
  [b_1, b_2b_3] &= b_1b_2b_3b_1^{-1}b_3^{-1}b_2^{-1} \\
  &= b_1b_2b_1^{-1} [b_1, b_3] b_2^{-1} \\
  &= [b_1, b_2] |b_1, b_3] 
  \end{align*}
  where the last line follows since $[b_1, b_3] \in A$ which is central in $B$. So the map $(b_1, b_2) \mapsto [b_1, b_2]$ is a $\ZZ$-bilinear map from $C$ to $A$, which has to be trivial since $C$ is torsion and $A$ is torsion-free. This means that $B$ is abelian. 
    
\end{proof}
 
\end{proof}

\begin{proof}[Proof of Proposition \ref{small intersection}]
  Let $\eps$ be the constant $\eps_G$ given by the Theorem \ref{AML} and $d = [k:\QQ]$. Let 
  \begin{align*}
    F &= \{ \gamma \in \Gamma :\: g \in S^{\eps d}_{\gamma} \} \\
    &= \{\gamma \in \Gamma :\: d_G(g^{-1}\gamma_i g, 1) < \eps d \}
  \end{align*}
  Let $\Delta = g\langle F\rangle g^{-1}$. We have 
  \begin{align*}
  |F| =: N &= |\{ \gamma \in \Delta :\: d_G(\gamma, 1) \le \eps d  \}| \\
  &= |\{ \gamma \in \Delta :\: d_X(\gamma x_0, x_0) \le \eps d  \}| 
  \end{align*}
  (where $x_0$ is the identity coset in $X = G/K$). We need to estimate $N$. For this we first observe that by Theorem \ref{AML} and Lemma \ref{bdd_index_abelian} there is an abelian subgroup $H$ of $\Delta$ with $[\Delta : H] \le A$. Let $L$ be a maximal free abelian subgroup of $H$.  Since $L \subset \Gamma$ which is a uniform lattice and $L$ is torsion-free, $L$ can contain only non-elliptic semisimple elements. By the Flat Torus Theorem \cite[Theorem 7.1 in Chapter II]{BriHa} there exists a symmetric subspace $Y \subset X$ and a flat subspace $E \subset X$ (of dimension rank $L$) such that the product $Y \times E$ is embedded in $X$, preserved by $L$ which acts trivially on the $Y$ factor and cocompactly on the $E$ factor. In particular, $L$ preserves the flat subspace $E \subset X$ itself. 
  
  Let $p$ be the geodesic projection from $X$ to $E$. It decreases distances, so 
  \[
  N \le |\{ \gamma \in \Delta :\: d_E(p(\gamma x_0), p(x_0)) \le \eps d  \}|. 
  \]
  Moreover, $\Delta$ normalizes a finite-index subgroup of $L$ so it preserves $E$. Therefore,  $p(\gamma x_0) = \gamma p(x_0)$ for any $\gamma \in \Delta$ and 
  \[
  N \le |\{ \gamma \in \Delta :\: d_E(\gamma x_1, x_1) \le \eps d  \}|
  \]
  for $x_1 := p(x_0) \in E$. 

  Let $\pi\colon\Delta\to \mathrm{Isom}(E)$ be natural map induced by the action of $\Delta$ on $E$. Let $T$ be the torsion subgroup of $H$; then as $L$ acts faithfully on $E$ we have that 
  \begin{equation} \label{ker_action}
    |\ker(\pi)| \le [\Delta:L] = |T| \cdot [\Delta : H]. 
  \end{equation}
  Let $\rho\colon \G\to \GL_n$, $n=\dim(G)$ be the adjoint representation. The group $T$ is abelian and consists only of torsion elements so $\rho(T)$ can be simultaneously diagonalized over the algebraic closure. Moreover, the diagonal entries are the adjoint eigenvalues of elements of $T$. As such, they are algebraic integers of degree at most $\dim G$ over $k$. The degree of the $D$-th cyclotomic polynomial over $\QQ$ is equal to $\phi(D) \ge \sqrt D$, so a root of unity which has degree $\dim(G) \cdot [k:\QQ]$ over $\QQ$ is of order at most $\dim(G)^2 [k:\QQ]^2$. In this way we can restrict the number of possible values on the diagonal entries by $C_0 [k:\QQ]^{m_0}$ for some $C_0,m_0$ dependent only on $G$. We deduce a crude (but sufficient) upper bound on the cardinality of $T$:
  \begin{equation} \label{tors_subgroup}
    |T| \le C_1 d^{m_1}
  \end{equation}
  with $C_1, m_1$ depending only on $G$.  

  If follows from \eqref{ker_action} and \eqref{tors_subgroup} that we have
  \begin{equation} \label{reduc_to_euc}
  N \le A \cdot C_1 d^{m_1} \cdot |\{ \gamma \in \pi(\Delta) :\: d_E(\gamma x, x) \le \eps d\}|. 
  \end{equation}
  We now use the following classical lemma, which we will prove after concluding our main argument. 

  \begin{lemma} \label{euclidean}
    Let $E$ be a Euclidean space, $\Lambda$ a lattice in $\mathrm{Isom}(E)$ and $x \in E$. If $\ell$ is the smallest length of a vector in the translation lattice of $\Lambda$ then for any $R \ge \ell$ we have 
    \[
    \{ g \in \Lambda :\: d_E(gx, x) \le R \} \le C_e \frac {R^{\dim(E)}}{\min\{1, \ell^{\dim(E)}\}}
    \]
    where $C_e$ is a constant depending only on $\dim(E)$. 
  \end{lemma}

  To conclude we observe that the translation lattice in $\pi(\Delta)$ is $\pi(L)$, and by Proposition \ref{dobrowolski_bound_injrad} we have that $\ell \ge \tfrac c{\log(d)^3}$. It then follows from Lemma \ref{euclidean} that  
  \[
  |\{ \gamma \in \pi(\Delta) :\: d_E(\gamma x, x) \le \eps d\}| \le \frac{C_e}c \cdot \log(d)^{3\dim(E)} \cdot (\eps d)^{\dim(E)}
  \]
  which implies that 
  \[
  |\{ \gamma \in \pi(\Delta) :\: d_E(\gamma x, x) \le \eps d\}| \le C_2 d^{m_2}
  \]
  for some $m_2, C_2$ depending only on $G$. Together with \eqref{reduc_to_euc} this finishes the proof of the proposition. 
\end{proof}

\begin{proof}[Proof of Lemma \ref{euclidean}]
  Let $r = \dim(E)$, let $L$ denote the translation lattice of $\Lambda$ and let $v_1, \ldots, v_r$ be the successive length minima for $L$. Let $D$ be the parallelepiped on $v_1, \ldots, v_r$ centered on $x$. Then $D$ is a fundamental domain for $L$ acting on $E$, so it contains at least one point from each orbit of $L$ on $\Lambda x$. Moreover, the index $[\Lambda:L]$ is bounded by a constant depending only on the dimension as follows from Bieberbach's theorem, so we need only to estimate
  \[
  |\{ \gamma \in L : \gamma D \cap B_E(x, R) \not= \emptyset \}|. 
  \]
  We will prove that if $\gamma D \cap B_E(x, R)$ is nonempty then $\gamma D \cap B_E(x, R+1)$ has volume at least $c_e \ell^r$ where $c_e$ is a constant depending only on the dimension; the lemma follows immediately by a packing argument. 

  For this it suffices to give a lower bound for the volume of $D \cap B_E(y, 1)$ where $y \in D$. We note that the simplex $S$ on the vectors $\tfrac {\eps_i v_i}{2\max(1,\|v_i\|)}$ based at $y$ lies entirely in $D \cap B_E(y, 1)$ (where $\eps_i$ are some signs). The volume of this simplex is equal to 
  \[
  \vol(S) = \frac 1{2^r r!} \prod_{i=1}^r \frac 1{\max(1,\|v_i\|)} \det(v_1, \ldots, v_r) = \frac 1{2^r r!} \cdot \frac{\vol(D)}{\prod_{i :\: \|v_i\| \ge 1} \|v_i\|}. 
  \]
  Now by Minkowski's second theorem we have that $\vol(D) \ge k_e \prod_{i=1}^r \|v_i\|$ (for some $k_e$ depending only on the dimension), so it follows that 
  \[
  \vol(S) \ge \frac{k_e}{2^r r!} \prod_{i:\: \|v_i\| \le 1} \|v_i\| \ge c_e \min(1, \ell)^r 
  \]
  where $c_e = k_e/2^r r!$, which finishes the proof of the lemma. 
\end{proof}

\begin{corollary} \label{tubes_vs_thin}
Let $\Gamma$ be an irreducible arithmetic lattice in $G$, $M = \Gamma \bs X$ and $k$ the trace field of $\Gamma$. For any $R \le \eps[k:\QQ]$ we have 
\[
\sum_{[\gamma] \subset \Gamma} \vol\left(\Gamma_\gamma \bs S_\gamma^R \right) \ge \vol(M_{\le R}) \ge \frac 1{[k:\QQ]^m} \sum_{[\gamma] \subset \Gamma} \vol\left(\Gamma_\gamma \bs S_\gamma^R \right)
\]
where the sum is over nontrivial conjugacy classes in $\Gamma$. 
\end{corollary}

\begin{proof}
Let $\Phi$ be the natural map 
\[
\Phi : \bigcup_{[\gamma] \subset \Gamma} \Gamma_\gamma \bs S_\gamma^R \to \Gamma \bs X = M. 
\]
We have that $M_{\le R}$ is the image of $\Phi$, and the volume measure on $M$ is locally the pushforward of the Haar measure on the right. Thus the left-hand inequality is immediate, and to prove the right-hand one it suffices to prove that $\Phi$ is at most $[k:\QQ]^m$-to-one. 

Fix $\gamma \in \Gamma \setminus\{1\}$ and $g \in S_\gamma^R$. If $\Phi(\Gamma_\theta g) = \Phi(\Gamma_\gamma h)$ for some $h\in S_\theta^R$ with $\theta,\in \Gamma \setminus \{1\}$ then $g$ lies in $S_{\eta\theta\eta^{-1}}^R$ for some $\eta \in \Gamma$, and by Proposition \ref{small intersection} the number of $\theta$ and for which this can happen is at most $[k:\QQ]^m$. 
\end{proof}

We can now prove the main result of this section, which will imply almost immediately Theorem \ref{Main_thinpart}. 

\begin{theorem}\label{thm-ThinPartsComparison}
  Let $G$ be a semisimple Lie group with the associated symmetric space $X$. There are constants $C, \delta, \eps, m > 0$ depending only on $G$ with the following property. Let $\Gamma\subset G$ be an irreducible arithmetic lattice with trace field $k$. For any $R_1, R_2 \geq 0$ with $R_1 + R_2 \leq \eps [k:\QQ]$ we have:
  \[
  \vol((\Gamma\bs X)_{\leq R_1}) \leq C e^{-\delta R_2}[k:\QQ]^m  \vol((\Gamma\bs X) _{\leq R_1 + R_2}).
  \]
\end{theorem}

\begin{proof}
By definition of $S_\gamma^R$ we have, for any $R >0$ and semisimple $\gamma \in \Gamma$, that
\[
\vol(\Gamma_\gamma\bs S_\gamma^{R})=\int_{\Gamma_\gamma\bs G}1_{\B(R)}(g^{-1}\gamma g)dg=\vol(\Gamma_\gamma\bs G_\gamma)O(\gamma,1_{\B(R)}).
\]
With Theorem \ref{thm-UniformGrowth} if follows that for any $R_1, R_2 > 0$ we have
\[
\vol(\Gamma_\gamma\bs S_\gamma^{R_1+R_2}) C e^{-\delta R_2} \ge  \vol(\Gamma_\gamma\bs S_\gamma^{R_1}) 
\]
and by Corollary \ref{tubes_vs_thin} it follows that 
\begin{align*}
  \vol(M_{\le R_1+R_2}) & \ge \frac 1{[k:\QQ]^m} \sum_ {[\gamma] \subset \Gamma} \vol(\Gamma_\gamma\bs S_\gamma^{R_1+R_2}) \\
  &\ge \frac 1{[k:\QQ]^m} \sum_ {[\gamma] \subset \Gamma} C e^{\delta R_2} \vol(\Gamma_\gamma\bs S_\gamma^{R_1}) \\
  &\ge \frac {C e^{\delta R_2}}{[k:\QQ]^m} \vol(M_{\le R_1}). 
\end{align*}
The last inequality proves the theorem. 
\end{proof}

\subsection{Proof of Theorem \ref{Main_thinpart}}
Let $\eta:=\eps/2$, $R_1=R_2=\eta[k:\QQ]$. We have $\vol((\Gamma\bs X) _{\leq R_1 + R_2}) \le \vol(\Gamma\bs X)$ so Theorem \ref{thm-ThinPartsComparison} yields the conclusion of Theorem \ref{Main_thinpart} in the case of uniform lattices. For non-uniform lattices the statement is empty, as the degree $[k:\QQ]$ is bounded. 


\section{Homotopy type}

Suppose $M$ is a manifold locally isometric to $X$ (i.e. $M=\Gamma\bs X$). For any $x \in M$ we denote
\[
\eps_x = \min(1, \InjRad_M(x)).
\]
In this section we will prove the following result.

\begin{proposition} \label{simplicial_general}
  Let $X$ be a symmetric space of non-compact type without Euclidean factors. There exists $A, B$ depending only on $X$ such that any compact $X$-manifold is homotopy equivalent to a simplicial complex with at most
  \[
  A \int_M \eps_x^{-\dim(X)} dx
  \]
  vertices and where every vertex belongs to at most $B$ simplices. 
\end{proposition}

Gelander's conjecture follows immediately from this and Theorem \ref{Main_thinpart}. 

\begin{proof}[Proof of Theorem \ref{Gelander_conj}] Indeed, let $M = \Gamma\bs X$ with $\Gamma$ a uniform, torsion-free arithmetic lattice in $G$. Then by the proposition we have a homotopy equivalence between $M$ and a simplicial complex with at most $N$ vertices where
\begin{align*}
N &\ll_X \int_M \eps_x^{-\dim(X)} dx \\
  &\le \InjRad(M)^{-\dim(X)} \vol(M_{\le 1}) + \vol(M) \\
  &\ll_X \vol(M)\left( e^{-c[k:\QQ]}\log[k:\QQ]^{3\dim(X)} + 1 \right) 
\end{align*}
where the last line is deduced using Dobrowolski's bound (Proposition \ref{dobrowolski_bound_injrad}) and Theorem \ref{Main_thinpart}. The term $e^{-c[k:\QQ]}\log[k:\QQ]^{-3\dim(X)}$ is bounded independently of $[k:\QQ]$ so this proves that $|S| = O(\vol(M))$ and finishes the proof of Theorem \ref{Gelander_conj}.
\end{proof}

\medskip

Before proving Proposition \ref{simplicial_general} we need a few preliminaries. We define, for $r \in [1, +\infty[$: 
\[
C_X(r) = \sup_{1/2 > t > 0} \frac {\vol_M (B_X(x_0, rt))} {\vol_M (B_X(x_0, t))}
\]
where $x_0$ is an arbitrary point in $X$. We say that a subset in a metric space is $\eta$-separated if all distances between distinct points are $> \eta$. The constant $C_X(r)$ will appear through the following packing lemma which we will use repeatedly in our arguments below. 

\begin{lemma} \label{packing}
  Let $M$ be an $X$-manifold, $0 < \eta < \delta < 1/2$ and $P$ an $\eta$-separated subset of $B_M(x, \delta)$. Then, $|P| < C_X(\delta/\eta)$. 
\end{lemma}

The proof is omitted. The main ingredient for the proof of Proposition \ref{simplicial_general} is the following lemma. 

\begin{lemma} \label{goodballs}
  Let $M$ be an $X$-manifold. 
  There exists a subset $S \subset M$ such that :
  \begin{enumerate}
  \item \label{dense} For every $x \in M$ there is a $p \in S$ such that $x \in B_M(p, \eps_p/6)$ ;

  \item \label{separated} If $p, q \in S$ and $p \not= q$ then $d(p, q) > \max(\eps_p, \eps_q)/36$.
  \end{enumerate}
  Moreover such a set $|S|$ satisfies 
  \begin{equation} \label{cardinality}
    |S| \ll_X \int_M \eps_x^{-\dim(X)} dx. 
  \end{equation}
\end{lemma}

\begin{proof}
  Such a set is given by any maximal subset which satisfies condition \eqref{separated} so we have only to prove \eqref{cardinality}. For $p \in S$ we define
  \[
  F_p(x) = \begin{cases} \vol_M (B_M(p, \eps_p/2))^{-1} \text{ if } x \in B_M(p, \eps_p/4) \\ 0 \text{ otherwise} \end{cases}
  \]
  and
  \[
  H_p(x) = \begin{cases} \vol_M(B_M(x, \eps_x/2))^{-1} \text{ if } p \in B_M(x, \eps_x/2) \\ 0 \text{ otherwise.} \end{cases} 
  \]
  Now if $x \in B_M(p, \eps_p/4)$ then $\eps_x > \eps_p/2$ (since the ball of radius $\eps_p/2$ around $x$ is contained in the ball of radius $\eps_p$ around $p$ and the latter is embedded). This implies that
  \[
  F_p(x) \le H_p(x) \cdot \frac {\vol_M (B_M(p, \eps_p))} {\vol_M (B_M(p, \eps_p/2))} \le C_X(2) H_p(x)
  \]
  for all $p, x$. 

  On the other hand, if $p \in B_M(x, \eps_x/2)$, then $\eps_p \ge \eps_x/2$, so that the subset $\{q \in S : H_q(x) \not= 0 \}$ is $\eps_x/72$-separated by \eqref{separated}. By Lemma \ref{packing} implies the inequality
  \[
  |\{q \in S : H_q(x) \not= 0 \}| \le C_X(36).
  \]
  So, we get that
  \[
  \sum_{p \in S} F_p(x) \le C_X(2) \sum_{p \in S} H_p(x) \le C_X(2)C_X(36) \vol_M(B_M(x, \eps_x/2))^{-1}
  \]
  for all $x \in M$. Now we have $\int_M F_p(x) dx = 1$, so we get
  \[
  |S|  = \int_M \sum_{p \in S} F_p(x) dx \le C_X(2)C_X(36) \int_M \vol_M(B_M(x, \eps_x/2))^{-1} dx. 
  \]
  On the other hand $\vol B_M(x, \eps_x/2) \asymp_X \eps_x^{-\dim(X)}$ since $\eps_x \le 1$, and \eqref{cardinality} follows. 
\end{proof}

\begin{proof} [Proof of Proposition \ref{simplicial_general}]
  We construct a simplicial complex by the same arguments as in \cite{Gelander1}, \cite[10.1]{Fraczyk}. A {\em good cover} of a topological space $Z$ is a collection $\mathcal U = \{U_1, \ldots, U_n\}$  of open subsets of $Z$ such that $Z = \bigcup_{i=1}^N U_i$ and all the nonempty intersections between the $U_i$ are contractible. If $\mathcal U$ is any collection of subsets of $Z$, its {\em nerve} $N(\mathcal U)$ is the simplicial complex with vertex set $\mathcal U$ and $U_0, \ldots, U_m$ is an $m$-simplex whenever $U_0 \cap \cdots \cap U_m$ is nonempty. By \cite[Theorem 13.4]{BotT}, if $\mathcal U$ is a good cover of $Z$ then $N(\mathcal U)$ is homotopy equivalent to $Z$. 

  Let $M$ be a $X$-manifold. Pick $S \subset M$ as in Lemma \ref{goodballs}. It follows from \eqref{dense} that the balls $B_M(p, \eps_p/6), p \in S$ are a cover of $M$, and each of them is contractible since $\eps_p/6 < \InjRad_M(p)$. Moreover, if $p, q \in S$ then $B_M(p, \eps_p/6) \cap B_M(q, \eps_q/6)$ is connected: assume otherwise and that $\eps_q \le \eps_p$. Choose a lift $\tilde q$ of $q$ to $X$. Then $B_X(\tilde q, \eps_q/6)$ intersects two distinct lifts of $B_M(p, \eps_p/6)$, with centers $\tilde p_1, \tilde p_2$. We have
  \[
  d_X(\tilde p_1, \tilde p_2) \le \frac{\eps_p}6 + 2\frac{\eps_q}6 + \frac{\eps_p}6 \le \frac{2\eps_p}3
  \]
  which contradicts the fact that $\eps_p \le \InjRad_M(p)$. 

  It follows that all the intersections between the balls $B(p, \eps_p/6), p \in S$ are isometric to intersections of balls of $X$, which are contractible. We conclude that $\mathcal U = \{B_M(p, \eps_p/6):\: p \in S\}$ is a good cover of $M$, so $M$ is homotopy equivalent to the simplicial complex $N(\mathcal U)$ which has at most $|S|$ vertices and by \eqref{cardinality} this implies the first part of the proposition. 

  It remains to bound the number of simplices incident to any vertex. To do this it suffices to bound the degree of a vertex $p$ in the 1-skeleton of the nerve of the cover $B_M(p, \eps_p/6), p \in S$. This follows from the same argument as above: if $q \in B_M(p, \eps_p/6)$ then $\eps_q > 5\eps_p/6$, so all points in $S \cap B_M(p, \eps_p/6)$ are $5\eps_p/216$-separated and applying Lemma \ref{packing} it follows that $B_M(p, \eps_p/6)$ contains at most $C_M(36/5)$ points of $S$. 
\end{proof}


\subsection{Triangulations} \label{delaunay}

Though it is unnecessary for most applications, it is natural to ask whether it is possible to replace ``homotopy equivalent'' by ``homeomorphic'' in the statement of Theorem \ref{Gelander_conj} or Proposition \ref{simplicial_general}. If $M$ has constant curvature (i.e. if $M$ is hyperbolic) then a set of points such as the one used in the argument explained above can also be used to construct a ``Delaunay triangulation'' of $M$ under a mild genericity assumption which we can always arrange for the sets constructed for the proof of the theorem. In variable curvature it is much more complicated to construct Delaunay triangulations but this was achieved in \cite{BDG_trig}; in this reference the authors describe an algorithm that, staring from a set of points such as that used for the proof of Gelander's conjecture, outputs a modified set of points from which a Delaunay triangulation can be constructed by patching local Euclidean triangulations together, and then show that it is homeomorphic to the manifold by using center of mass maps from the simplices to the manifold. Unfortunately in this reference the authors use nets with constant separation, which cannot be applied to our situation since the delicate case is exactly when the injectivity radius might go to 0. Their construction works in general as proven in their later work  \cite{BDGW}. 


\subsection{Lower bounds for the minimal number of simplices} \label{lowbd_simpl}

The bounds given by Theorem \ref{Gelander_conj} are qualitatively sharp, as shown by the following proposition whose proof follows an argument shown to us by Nir Lazarovich for hyperbolic manifolds, which generalises using the straightening construction of Lafont and Schmidt \cite{Lafont_Schmidt}\footnote{Note that an argument using simplicial volume as a black box does not seem to work here as it is not clear how to control the representation of a fundamental class in a simplicial complex homotopy equivalent to a closed manifold. }. A similar result for word-hyperbolic groups is proven in \cite{Lazarovich_complexity}. 

\begin{proposition}
  Let $X$ be a symmetric space of non-compact type without Euclidean factors. There exists a constant $c_X > 0$ such that for any closed manifold $M$ locally isometric to $X$ and any simplicial complex $C$ which is homotopy equivalent to $X$, $C$ has at least $c_X\vol(M)$ simplices of dimension $\dim(X)$. 
\end{proposition}

\begin{proof}
  We assume that $X$ is not isometric to $\HH^2$ (for which the result follows immediately from Gauss-Bonnet). Let $n = \dim(X)$. Let $f : C \to M$ be a homotopy equivalence. We may assume that $f$ is a simplicial map to some triangulation of $M$. Under these conditions it is immediate that the restriction of $f$ to $n$-simplices is surjective: indeed, let $\gamma$ be a simplicial $n$-chain in $C$ representing the fundamental class (mod 2 if $M$ is non-orientable) of $M$; then the union of the $f$-images of the simplices in $\gamma$ must cover $M$. Let $Y$ be the universal cover of the $n$-skeleton $C^{(n)}$ and let $\tilde f : Y \to X$ be the lift of $f$, which is a surjective, $\Gamma$-equivariant map ($\Gamma = \pi_1(M)$). 

  Let $s$ be the straightening map on singular simplices of $X$ given by \cite[Definition p.~ 136]{Lafont_Schmidt} (see \cite{Bucher} for the case where $X = \SL_3(\RR)/\SO(3)$). If $\Delta$ is a $n$-simplex of $Y$ then $\tilde f|_\Delta$ is a singular simplex of $X$ so we can define a map $\tilde g_\Delta = s(f|_\Delta)$. Since $s$ is a $\Gamma$-equivariant chain map the $\tilde g_\Delta$ glue together to give a surjective $\Gamma$-equivariant map $\tilde g : Y \to X$. Let $g : C^{(n)} \to M$ be the quotient map; it is surjective and since the siplicies are straight the volume of the image of an $n$-simplex of $C$ is bounded above by a constant $1/c$ depending only on $X$ (see property (d) in \cite[pp.~132-133]{Lafont_Schmidt}). Because of surjectivity of $g$ the total volume of the images of these simplices must be at least $\vol(M)$, hence there are at least $c\vol(M)$ of them. 
\end{proof}

It may still be possible to construct CW-complexes with few cells which are homotopy equivalent to locally symmetric space of large volume. For instance, in \cite{ABFG} the first author together with M.~Abért, N.~Bergeron and D.~Gaboriau construct CW-complexes which are homotopy equivalent to congruence covers of certain locally symmetric spaces and have relatively few cells of low dimensions compared to the degree.

\bibliographystyle{plain}
\bibliography{bib}
\end{document}

%% file: macros.tex
\DeclareFontFamily{U}{wncy}{}
\DeclareFontShape{U}{wncy}{m}{n}{<->wncyr10}{}
\DeclareSymbolFont{mcy}{U}{wncy}{m}{n}
\DeclareMathSymbol{\Sha}{\mathord}{mcy}{"58}

\newcommand{\eps}{\varepsilon}

\newcommand{\bs}{\backslash}

\newcommand{\vol}{\operatorname{vol}}

\newcommand{\tr}{\operatorname{tr}}

\newcommand{\ad}{\operatorname{Ad}}

\newcommand{\interval}[4]{
  \ifthenelse{ \equal{#1}{o} } {\mathopen{]}} {\mathopen{[}}
  #2, #3
  \ifthenelse{ \equal{#4}{o} } {\mathclose{[}} {\mathclose{]}}
}

\newcommand{\T}{\mathbf{T}}
\renewcommand{\H}{\mathbf{H}}

\newcommand{\G}{\mathbf{G}}

\newcommand{\SL}{\mathrm{SL}}
\newcommand{\SU}{\mathrm{SU}}
\newcommand{\SO}{\mathrm{SO}}

\newcommand{\GL}{\mathrm{GL}}

\newcommand{\bG}{\mathbf{G}}
\newcommand{\bH}{\mathbf{H}}
\newcommand{\bN}{\mathbf{N}}

\newcommand{\CC}{\mathbb C}
\newcommand{\RR}{\mathbb R}
\newcommand{\ZZ}{\mathbb Z}
\newcommand{\HH}{\mathbb H}

\newcommand{\QQ}{\mathbb Q}

\newcommand{\InjRad}{\operatorname{InjRad}}

\newcommand{\B}{\mathrm{B}}

%% file: mainrev.bbl
\begin{thebibliography}{10}

\bibitem{ABFG}
Miklos Abert, Nicolas Bergeron, Miko{\l}aj Fr{\c{a}}czyk, and Damien Gaboriau.
\newblock On homology torsion growth, 2025.

\bibitem{Bader_Gelander_Sauer}
Uri {Bader}, Tsachik {Gelander}, and Roman {Sauer}.
\newblock {Homology and homotopy complexity in negative curvature}.
\newblock {\em {J. Eur. Math. Soc. (JEMS)}}, 22(8):2537--2571, 2020.

\bibitem{BallGro}
Werner Ballmann.
\newblock Manifolds of non positive curvature.
\newblock In {\em Arbeitstagung Bonn 1984}, pages 261--268. Springer, 1985.

\bibitem{BekkaValette}
Bachir Bekka, Pierre de~La~Harpe, and Alain Valette.
\newblock {\em Kazhdan's property (T)}.
\newblock Cambridge university press, 2008.

\bibitem{BDG_trig}
Jean-Daniel {Boissonnat}, Ramsay {Dyer}, and Arijit {Ghosh}.
\newblock {Delaunay triangulation of manifolds}.
\newblock {\em {Found. Comput. Math.}}, 18(2):399--431, 2018.

\bibitem{BDGW}
Jean-Daniel Boissonnat, Ramsay Dyer, Arijit Ghosh, and Mathijs Wintraecken.
\newblock Local criteria for triangulating general manifolds, 2023.

\bibitem{BorelBook}
Armand Borel.
\newblock {\em Linear algebraic groups}, volume 126.
\newblock Springer Science \& Business Media, 2012.

\bibitem{BHC62}
Armand Borel and Harish-Chandra.
\newblock Arithmetic subgroups of algebraic groups.
\newblock {\em Ann. of Math. (2)}, 75:485--535, 1962.

\bibitem{Borel_Prasad}
Armand Borel and Gopal Prasad.
\newblock Finiteness theorems for discrete subgroups of bounded covolume in
  semi-simple groups.
\newblock {\em Publications Math\'ematiques de l'IH\'ES}, 69:119--171, 1989.

\bibitem{BotT}
R.~Bott and L.W. Tu.
\newblock {\em Differential Forms in Algebraic Topology}.
\newblock Graduate Texts in Mathematics. Springer New York, 1995.

\bibitem{Breuillard}
Emmanuel Breuillard.
\newblock A height gap theorem for finite subsets of {${\rm GL}\sb
  d(\overline{\mathbb Q})$} and nonamenable subgroups.
\newblock {\em Ann. of Math. (2)}, 174(2):1057--1110, 2011.

\bibitem{BriHa}
Martin~R. Bridson and Andr{\'e} Haefliger.
\newblock {\em Metric spaces of non-positive curvature}, volume 319 of {\em
  Grundlehren Math. Wiss.}
\newblock Berlin: Springer, 1999.

\bibitem{Bucher}
Michelle Bucher-Karlsson.
\newblock Simplicial volume of locally symmetric spaces covered by
  {{\(\mathrm{SL}_{3} \mathbb R/\mathrm{SO}(3)\)}}.
\newblock {\em Geom. Dedicata}, 125:203--224, 2007.

\bibitem{chen2021height}
Lvzhou Chen, Sebastian Hurtado, and Homin Lee.
\newblock A height gap in {{\(\mathrm{GL}_d(\overline{\mathbb{Q}})\)}} and
  almost laws.
\newblock {\em Groups Geom. Dyn.}, 19(3):899--912, 2025.

\bibitem{Cor1}
Kevin Corlette.
\newblock Hausdorff dimensions of limit sets i.
\newblock {\em Inventiones mathematicae}, 102(1):521--541, 1990.

\bibitem{Dobro79}
E.~Dobrowolski.
\newblock On a question of lehmer and the number of irreducible factors of a
  polynomial.
\newblock {\em Acta Arith.}, 34(4):391--401, 1979.

\bibitem{Els1}
J{\"u}rgen Elstrodt.
\newblock Die resolvente zum eigenwertproblem der automorphen formen in der
  hyperbolischen ebene. teil i.
\newblock {\em Mathematische Annalen}, 203(4):295--330, 1973.

\bibitem{Els2}
J{\"u}rgen Elstrodt.
\newblock Die resolvente zum eigenwertproblem der automorphen formen in der
  hyperbolischen ebene. teil ii.
\newblock {\em Mathematische Zeitschrift}, 132(2):99--134, 1973.

\bibitem{Els3}
J{\"u}rgen Elstrodt.
\newblock Die resolvente zum eigenwertproblem der automorphen formen in der
  hyperbolischen ebene. teil iii.
\newblock {\em Mathematische Annalen}, 208(2):99--132, 1974.

\bibitem{fisher2022new}
David Fisher and Sebastian Hurtado.
\newblock A new proof of finiteness of maximal arithmetic reflection groups.
\newblock {\em Ann. Henri Lebesgue}, 6:151--159, 2023.

\bibitem{Fraczyk}
M.~Fr{\c a}czyk.
\newblock Strong limit multiplicity for arithmetic hyperbolic surfaces and
  3-manifolds.
\newblock {\em Invent. math.}, 2020.

\bibitem{frachurtraim}
Mikolaj Fraczyk, Sebastian Hurtado, and Jean Raimbault.
\newblock {B}enjamini--{S}chramm convergence of arithmetic locally symmetric
  spaces.
\newblock {\em In preparation}, 2023.

\bibitem{Gelander1}
Tsachik Gelander.
\newblock Homotopy type and volume of locally symmetric manifolds.
\newblock {\em Duke Math. J.}, 124(3):459--515, 2004.

\bibitem{Gelander2}
Tsachik Gelander.
\newblock Volume versus rank of lattices.
\newblock {\em Journal fur die reine und angewandte Mathematik (Crelles
  Journal)}, 2011(661):237--248, 2011.

\bibitem{gelander2021bounds}
Tsachik Gelander and Paul Vollrath.
\newblock Bounds on systoles and homotopy complexity.
\newblock {\em arXiv preprint arXiv:2106.10677}, 2021.

\bibitem{GelaVol}
Tsachik Gelander and Paul Vollrath.
\newblock Bounds on {Systoles} and {Homotopy} {Complexity}.
\newblock Preprint, {arXiv}:2106.10677 [math.{GR}] (2021), 2021.

\bibitem{Lafont_Schmidt}
Jean-Fran\c{c}ois {Lafont} and Benjamin {Schmidt}.
\newblock {Simplicial volume of closed locally symmetric spaces of non-compact
  type}.
\newblock {\em {Acta Math.}}, 197(1):129--143, 2006.

\bibitem{Lapan_Linowitz_Meyer}
Sara Lapan, Benjamin Linowitz, and Jeffrey~S. Meyer.
\newblock Systole inequalities up congruence towers for arithmetic locally
  symmetric spaces.
\newblock {\em Commun. Anal. Geom.}, 31(4):847--878, 2023.

\bibitem{Lazarovich_complexity}
Nir Lazarovich.
\newblock Finite index rigidity of hyperbolic groups, 2023.

\bibitem{Pat1}
S.~J. Patterson.
\newblock The limit set of a {F}uchsian group.
\newblock {\em Acta Math.}, 136(3-4):241--273, 1976.

\bibitem{platonov1968some}
Vladimir~Petrovich Platonov.
\newblock Some remarks on linear groups.
\newblock {\em Mathematical notes of the Academy of Sciences of the USSR},
  4(6):873--874, 1968.

\bibitem{raimbault2022coxeter}
Jean Raimbault.
\newblock Coxeter polytopes and {B}enjamini--{S}chramm convergence.
\newblock {\em Bulletin de la SMF}, 151, 2024.

\bibitem{Sul1}
Dennis Sullivan.
\newblock Related aspects of positivity in {R}iemannian geometry.
\newblock {\em J. Differential Geom.}, 25(3):327--351, 1987.

\bibitem{Varadarajan}
V.~S. Varadarajan.
\newblock {\em Harmonic analysis on real reductive groups}, volume 576 of {\em
  Lect. Notes Math.}
\newblock Springer, Cham, 1977.

\bibitem{Vinberg_trace}
E.~B. {Vinberg}.
\newblock {Rings of definition of dense subgroups of semisimple linear groups}.
\newblock {\em {Math. USSR, Izv.}}, 5:45--55, 1972.

\bibitem{Wehrfritz}
Bertram Wehrfritz.
\newblock {\em Infinite linear groups: an account of the group-theoretic
  properties of infinite groups of matrices}, volume~76.
\newblock Springer Science \& Business Media, 2012.

\end{thebibliography}
